\documentclass[a4paper]{article}
\usepackage[ragged]{footmisc}
\usepackage{amssymb,amsfonts,amsmath,mathrsfs,tikz,authblk,geometry}
\usepackage{comment}
\title{A class of trees determined by their chromatic symmetric functions}
\author[a,1]{Yuzhenni  Wang
}
\author[b,2]{Xingxing Yu}

\author[$*$,a,1]{Xiao-Dong Zhang}

\affil[a]{School of Mathematical Sciences, MOE-LSC,  SHL-MAC,
	Shanghai Jiao Tong University, Shanghai 200240, P. R. China
}

\affil[b]{School of Mathematical Sciences, Georgia Institute of Technology, Atlanta, GA 30332-0160, USA
}
\newenvironment {Proof of 1}{\noindent {\bf Proof of Theorem \ref{result1}.}}{\hfill\ensuremath{\square}}
\date{}
\usepackage{graphicx}	
\usepackage{amsthm}
\newtheorem{thm}{Theorem}[section]
\newtheorem{lem}[thm]{Lemma}
\newtheorem{coro}[thm]{Corollary}
\newtheorem{prop}[thm]{Proposition}
\newtheorem{conj}[thm]{Conjecture}

\newtheorem*{rem}{Remark}
\newtheorem{cla}{Claim}[]

\begin{document}
	\maketitle
	\renewcommand{\thefootnote}{}
	\footnote{*Corresponding author. }
	\footnote{E-mail address:
		wangyuzhenni@sjtu.edu.cn(Y. Wang),  yu@math.gatech.edu(X. Yu),  xiaodong@sjtu.edu.cn(X.-D. Zhang)
	}
	\footnote{$^1$Partly supported by the National Natural Science Foundation of China (Nos. 11971311,12371254, 12161141003) and Science and Technology Commission of Shanghai Municipality (No. 22JC1403600),  National Key R\&D Program of China under Grant No. 2022YFA1006400 and the Fundamental Research Funds for the Central Universities.
	}
	
	\begin{abstract}
		Stanley introduced the concept of chromatic symmetric functions of graphs which extends and refines the notion of chromatic polynomials of graphs, and asked whether trees are determined up to isomorphism by their chromatic symmetric functions.  
		Using the technique of differentiation with respect to power-sum symmetric functions, we give a positive answer to  Stanley's question for the class of trees with exactly two vertices of degree at least 3. 
		In addition, we prove that for any tree $T$, the generalized degree sequence for subtrees of $T$ is determined by the chromatic symmetric function of $T$, providing evidence to a conjecture of Crew.
		
	\end{abstract}
	\section{Introduction}

	We use  $\mathbb{Z}^+$ to denote the set of positive integers.
	For $n\in\mathbb{Z}^+$, define $[n]=\{1,\dots,n\}$.
	Let $G$ be a graph possibly with  loops or multiple edges.
	Let $V(G)$ and $E(G)$ denote the vertex set and edge
	set of $G$, respectively. Sometimes, we write $(V(G),
	E(G))$ for  $G$, and write $|G|$ for $|V(G)|$, and call it the {\it order} of $G$. For $S\subseteq E(G)$, we use $G-S$ to denote the graph $(V(G), E(G)\setminus S)$. For $S\subseteq V(G)$, we use $G[S]$ to denote the subgraph of $G$ induced by $S$,
	and  we use $G-S$ to denote  $G[V(G)\setminus S]$.  For two graphs $G$ and $H$, we write  $G\cong H$ for
	``$G$ is isomorphic to $H$''.

	A proper coloring of a graph $G$ is a function $\kappa: V(G) \rightarrow \mathbb{Z}^+$ such that $\kappa(v) \neq \kappa(w)$ whenever $vw\in E(G)$. Stanley (\cite{Stanley1995}; see also \cite[pp. 462--464]{Stanley1999}) defined the {\it chromatic symmetric function} of $G$ as
	$$X_G = X_G(x_1,x_2,\dots) = \sum_{\kappa}\prod_{v\in V(G)} x_{\kappa(v)},
	$$
	where $x_1,x_2,\dots$ are commuting variables, and
	the summation is taken  over all proper colorings $\kappa$ of $G$.
	Note that if there is a loop in $G$ then $X_G = 0$ (as in this case $G$ admits no proper coloring), and that $X_G$ remains the same if we replace multiple edges with a single edge.
	Also note that $X_G$ is invariant  under permutations of ${x_i}$;  so $X_G$ is a symmetric function, homogeneous of degree $|G|$. By definition, if $G$ is a graph with components $G_1,\dots,G_m$,
	then $$X_{G}=\prod_{i=1}^m X_{G_i}.$$
	
	By letting  $x_1 =\dots=x_k =1$ and $x_i =0$ for all $i>k$ in $X_G$, we obtain the  {\it chromatic polynomial} of $G$ in $k$, denoted by  $\chi_G(k)$, which counts the number of  proper colorings of $G$ using at most $k$ colors. For any tree $T$ on $n$ vertices,  $\chi_T(k)=k(k-1)^{n-1}$. Thus, the chromatic polynomial does not determine the isomorphism types of trees; however, Stanley  \cite{Stanley1995} asked whether chromatic symmetric functions do.
	It was proved that some classes of trees are determined by the chromatic symmetric functions, so Stanley's question is often formulated as a conjecture.
	
	\begin{conj}[Stanley's tree problem, 1995]\label{conj}
		For two trees $T$ and $F$, $X_{T}=X_{F}$ if
		and only if $T\cong F$.
	\end{conj}
	
	Note that  Conjecture~\ref{conj} need not hold for graphs containing cycles.  The two unicyclic graphs in Figure \ref{fig1} are not isomorphic but share the same chromatic symmetric function.  More such pairs can be found in  Orellana and Scott  \cite{Orellana2014}.
	\begin{figure}[ht]
		\centering
		\begin{tabular}{ccc}
			\scalebox{1}{
				\begin{tikzpicture}
					\foreach \x in {0,1,2}\fill(\x,0) circle (.08);
					\foreach \x in {0,1,2}\fill(\x,1) circle (.08);
					\draw [line width=0.6pt](0,0)--(0,1)--(1,0)--(1,1)--(2,1);
					\draw[line width=0.6pt](0,0)--(2,0);
			\end{tikzpicture}}
			&&\scalebox{1}{
				\begin{tikzpicture}
					\foreach \x in {0,1,2}\fill(\x,0) circle (.08);
					\foreach \x in {0,1,2}\fill(\x,1) circle (.08);
					\draw[line width=0.6pt] (0,1)--(0,0)--(1,1)--(1,0);
					\draw[line width=0.6pt](1,1)--(2,1);
					\draw[line width=0.6pt](0,0)--(2,0);
			\end{tikzpicture}}
		\end{tabular}
		\caption[]{\label{fig1}Two unicyclic graphs with the same chromatic symmetric function}
	\end{figure}
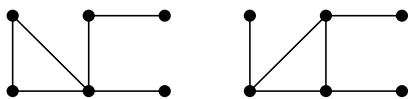

	Conjecture \ref{conj} holds for all trees with up to 29 vertices  \cite{Heil2019}, as well as for certain classes of trees
	\cite{Aliste-Prieto2017,Aliste-Prieto2014,Chmutov2020,Loebl2019}.
	In particular, Martin, Morin, and Wagner \cite{Martin2008} proved that Conjecture \ref{conj} holds for {\it spiders}, trees with exactly one vertex of degree at least 3.
	We show in this paper that Conjecture \ref{conj} holds for trees with exactly two vertices of degree at least 3.
	
	\begin{thm}\label{result1}
		Suppose $T$ is a tree with exactly two vertices of degree at least 3. If $F$ is a graph with $X_F=X_T$, then $F\cong T$.
	\end{thm}

	We prove Theorem~\ref{result1} by applying techniques from previous work  \cite{Crew2022,Crew2020,Martin2008,Stanley1995} and by considering expansions of chromatic symmetric functions.
	In Section 2,  we discuss useful properties of chromatic symmetric functions and their weighted version, and ways of expressing $X_G$ as a linear combination of chromatic symmetric functions of certain graphs obtained from  $G$.
	In Section 3, we complete the proof of Theorem \ref{result1}.
	In Section 4, we prove a result motivated by a conjecture of Crew concerning the generalized degree sequences of trees.

	\section{Preliminaries}
	Recall that, for a graph $G$, $X_G$ is a homogeneous symmetric function of degree $|G|$. For $n\in \mathbb{Z}^+$, let $\varLambda_n$ consist of all symmetric functions that are homogeneous of degree $n$.
	Then  $\varLambda_n$ is a vector space over the rationals $\mathbb{Q}$.
	
	For $n\in \mathbb{Z}^+$, a partition of $n$ is a sequence $\lambda=(\lambda_1,\dots,\lambda_l)$ of positive integers such that $\sum_{i=1}^l\lambda_i=n$ and  $\lambda_1\ge\dots\ge\lambda_l$,  often written as $\lambda\vdash n$.  
	For a partition $\lambda\vdash n$,  the length of $\lambda$ is the number of parts in $\lambda$,  denoted as $l(\lambda)$.
	For $k\in \mathbb{Z}^+$, the $k$-th power-sum symmetric function is
	$$p_k=\sum_{i\ge 1}x_i^k,$$
	where $x_1,x_2,\dots$ are commuting variables. For a partition $\lambda=(\lambda_1\dots,\lambda_l)$, let
	$$p_{\lambda}=p_{\lambda_1}\dots p_{\lambda_l}.$$

	It is well-known that $\{p_{\lambda}:\lambda\vdash n\}$ is a basis for the $\mathbb{Q}$-vector space $\varLambda_n$. Thus, there exist $c_\lambda(G)\in \mathbb{Q}$ such that
	$$X_G=\sum_{\lambda\vdash n}c_{\lambda}(G)p_\lambda.$$
	Stanley \cite[Theorem 2.5]{Stanley1995} proved that for any graph $G$,
	\begin{equation}\label{eq2}
		X_G=\sum_{A\subseteq E(G)}(-1)^{|A|}p_{\pi(A)},
	\end{equation}
	where $\pi(A)$ is the partition of $|G|$ whose parts are the orders of the components of the graph $(V(G),A)$.

	If $F$ is an $n$-vertex forest, then $l(\pi(A))+|A|=n$ for every subset $A\subseteq E(F)$.  Thus by \eqref{eq2},
	\begin{equation*}
		c_\lambda(F)=(-1)^{n-l(\lambda)}|\{A\subseteq E(F) : \pi(A)=\lambda\}|.
	\end{equation*}
	Hence, if $F$ is an $n$-vertex forest with $j$ components, then
	\begin{equation*}
		\sum_{\lambda : l(\lambda)=k}c_\lambda(F)=(-1)^{n-k}\binom{n-j}{k-j}.
	\end{equation*}

	A useful tool for studying symmetric functions is differentiation.
	By \eqref{eq2}, we may regard $X_G$ as a polynomial in $p_1,\dots,p_{|G|}$, but not in $x_1,x_2,\dots$. Let $\partial/\partial p_j$ denote the formal partial derivative of $X_G$ with respect to $p_j$.
	Stanley \cite[Corollary 2.11]{Stanley1995} proved that for any graph $G$ and any  $j\in\mathbb{Z}^+$,
	\begin{equation*}
		\frac{\partial X_G}{\partial p_j}=\sum_H\mu_H X_{G-V(H)},
	\end{equation*}
	where the summation is taken  over all connected $j$-vertex induced subgraphs $H$  of $G$, and $\mu_H$ is the coefficient of $n$ in the chromatic polynomial $\chi_H(n)$.
	Thus,   $\mu_H=(-1)^{j-1}$ for any $j$-vertex tree $H$
     (because  $\chi_H(n)=n(n-1)^{j-1}$), and we have the following
	\begin{lem}\label{prop3}
		For any forest $F$,
		\begin{equation*}
			\frac{\partial X_F}{\partial p_j}=(-1)^{j-1}\sum_H X_{F-V(H)},
		\end{equation*}
		where the summation is taken over all $j$-vertex subtrees $H$ of $F$.
	\end{lem}

Let $\text{Sym}_{\mathbb{Q}}$ denote the ring of the chromatic symmetric functions of all simple connected graphs.
Tsujie \cite[Corollary 2.4]{Tsujie2018} stated  that $X_G$ is irreducible in $\text{Sym}_{\mathbb{Q}}$ if and only if $G$ is connected, which also follows from  \cite[Theorem 5]{Cho2016}. As a consequence of this result, we have the following lemma.
\begin{lem}\label{thm-forest}
	Let $m,r\in \mathbb{Z}^+$. Let $T$ be a forest with components $T_1, \dots, T_m$ and $F$ be a forest with components  $F_1, \dots, F_r$. Suppose $X_T=X_F$.
	Then $m=r$, and there exists a permutation $\tau$ of $[m]$ such that $X_{T_i}=X_{F_{\tau(i)}}$ for $i\in [m]$.
\end{lem}
	
	Next, we consider  subtree polynomials.
	For any tree $T$, denote by $L(T)$ the set of {\it leaf edges} of $T$ (i.e., edges of $T$ incident with leaves of $T$). 
	The {\it subtree polynomial} of $T$ is
	$$S_T =S_T(q,r)=\sum_{S}q^{|E(S)|}r^{|L(S)|},$$
	where the summation is taken over all subtrees $S$ of $T$.
	Let $s_T(i,j)$ denote the coefficient of $q^ir^j$ in $S_T$, which is the number of subtrees of $T$ with $i$ edges and $j$ leaf edges. Then $s_T(i,2)$ is the number of paths of length $i$ in $T$. It is observed in  \cite{Chaudhary1991} that
        $$\sum_{k\ge i}\binom{k}{i}(-1)^{i+k}s_T(k,k)$$
        is the number of vertices of degree $i$ in $T$.  
        For a graph $G$, the {\it degree sequence} of $G$ is the sequence $(d_1, d_2, \dots, d_{|G|-1})$ 
        where $d_i$ is the number of vertices of degree $i$ in $G$,
	and the  {\it path sequence} of $G$ is the sequence $(s_1, s_2, \dots, s_{|G|-1})$ where $s_i$ is the number of paths of length $i$  in $G$. 
	Thus, the degree sequence and path sequence of a tree can be determined by its  subtree polynomial.

	Martin, Morin, and Wagner \cite{Martin2008} proved that the subtree polynomial $S_T$ of a tree $T$ is
	determined by  $X_T$. 
    More generally, they obtained the following formula for $S_T$ in terms of the usual scalar product $\langle\cdot,\cdot\rangle$ on $\varLambda_n$ where $n=|T|$ (see \cite[\S7.9]{Stanley1999}):
	$$S_T(q,r)=\langle\Phi_n(q,r),X_T\rangle,$$
	where $\Phi_n(q,r)$	does not depend on $T$. Thus the linearity of the scalar product implies
	\begin{prop}\label{prop-path}
		The degree sequence and path sequence of any tree $T$  are determined by $S_{T}$, and hence also by $X_{T}$. 
		In general, the degree sequence and path sequence of any forest $F$ with components $T_1,\dots, T_m$ of the equal order are determined by $S_{T_1}+\dots+S_{T_m}$, and hence also by $X_{T_1}+\dots+X_{T_m}$.
	\end{prop}

	Martin, Morin, and Wagner \cite{Martin2008} also proved that spiders are determined by their subtree polynomials, so Conjecture \ref{conj} holds for spiders.
	
	For a tree $T$, the {\it trunk} is defined to be its smallest subtree that contains all vertices of degree at least 3 in $T$. 
	So the trunk of a spider consists of a single vertex.  
	For every leaf $v$ of $T$, define its {\it twig}  to be the longest path in $T$ containing $v$ such that every internal vertex of $P$ has degree 2 in $T$.
	A path of $T$ is called a {\it twig} if it is a twig for one of the leaves of $T$. For a forest $F$,
	a twig of $F$ is a twig of some component of $F$.
	Define the {\it twig sequence} of a forest $F$ to be $(\tau_1,\tau_2,\dots,\tau_{|T|-1})$, where $\tau_i$ denotes the number of twigs of length $i$ in $F$.
	Crew \cite{Crew2022} generalized an argument in \cite{Martin2008} and proved the following.

	\begin{lem}\cite[Theorem 3]{Crew2022}\label{lem-trunk}
		For any tree $T$, the order of its trunk and the twig sequence are determined by  $S_T$, and thus also by $X_{T}$. 
		In particular, if $T$ is a spider then $T$ is determined by $X_T$.  
	\end{lem}
	
	Crew and Spirkl \cite{Crew2020} extended the concept of chromatic symmetric functions to vertex-weighted graphs, which enabled them to extend the deletion-contraction formula for chromatic polynomials to chromatic symmetric functions.
	Relevant results can be found in \cite{Aliste-Prieto2021,CS2022}.
	
	A weighted graph $(G,w)$ consists of a graph $G$ (possibly with multiple edges and loops) and a weight function $w:V(G)\rightarrow\{0\}\cup \mathbb{Z}^+$. 
	For any $U\subseteq V$, define $w(U)=\sum_{u\in U}w(u)$. 
	Two weighted graphs $(G,w)$ and $(H,w^\prime)$ are isomorphic if there is an isomorphism from $G$ to $H$, $\sigma: V(G) \to V(H)$, such that $w'(\sigma(v))=w(v)$ for all $v\in V(G)$.  
	Crew and Spirkl \cite{Crew2020} introduced the following generalization of chromatic symmetric functions
	\[
	X_{(G,w)}=X_{(G,w)}(x_1,x_2,\dots)=\sum_{\kappa}\prod_{v\in V(G)}x_{\kappa(v)}^{w(v)},
	\]
	where the summation is taken over all proper colorings $\kappa$ of $G$, and $x_1,x_2,\dots$ are commuting variables. 
	Note that for any graph $G$,
	$X_{(G,1_V)}=X_G$, where $1_V(v)=1$ for all $v\in V(G)$.  Clearly, $X_{(G,w)}$ is invariant under permutations of ${x_i}$, so it is a symmetric function and homogeneous of degree $w(V(G))$.

	For an edge $e=u_1u_2$ of $(G,w)$, let $(G-e,w)$ be the graph obtained from $G$ by deleting $e$ and preserving weights of all the vertices. 
	If $e$ is a loop (i.e., $u_1=u_2$) then let $G/e=G-e$ and preserve weights of all the vertices. 
	If $e$ is not a loop, then let $G/e$ denote the graph obtained from $G$ by contracting $e$, i.e., deleting $e$ and identifying $u_1$ and $u_2$ into a single vertex $u$. 
	Let $w_{G/e}: V(G/e)\to \mathbb{Z}^+\cup \{0\}$ such that $w_{G/e}(v)=w(v)$ for all $v\in V(G/e)\setminus \{u\}$ and
	$$w_{G/e}(u)=w(u_1)+w(u_2). $$
	For any edge $e$ (that can be a loop) of a vertex-weighted graph $(G,w)$,  Crew and Spirkl \cite{Crew2020} derived the following deletion-contraction formula:
		\begin{equation*}
			X_{(G,w)}=X_{(G-e,w)}-X_{(G/e,w_{G/e})}.
		\end{equation*}

	Crew and Spirkl \cite{Crew2020} also proved the $p_{\lambda}$-basis expansion of $X_{(G,w)}$, which is a weighted version of \eqref{eq2}: 
	for any vertex-weighted graph $(G,w)$,
	\begin{equation*}
		X_{(G,w)}=\sum_{A\subseteq E}(-1)^{|A|}p_{\pi(G,w,A)},
	\end{equation*}
	where $x_1,x_2,\dots$ are commuting variables, and $\pi(G,w,A)$ is the partition of $w(V(G))$ whose parts are the total weights of the vertex sets of the components of the graph $(V(G),A)$.

	  For any weighted graph $(G,w)$ and any  $S=\{e_1,\dots,e_k\}\subseteq E(G)$ with $k\ge 2$, define
	  $$w_{G/S}\equiv w_{(G/\{e_1,\dots,e_{k-1}\})/e_k}\equiv\dots\equiv w_{G/e_1/\dots/e_{k-1}/e_k}.$$
	  We point out that $w_{G/S}$ does not depend on the order of the edges being contracted. By repeatedly applying  the deletion-contraction formula, one obtains the following 
	\begin{prop}\label{thm5}
		For a vertex-weighted graph $(G,w)$ and a non-empty set $S\subseteq E(G)$,
		\begin{equation*}
			X_{(G,w)}=\sum_{\emptyset\neq I\subseteq S}(-1)^{|I|-1}X_{({G-I},w)}+(-1)^{|S|}X_{({G/S},w_{G/S})}
		\end{equation*}
	\end{prop}

	  Proposition \ref{thm5} restricted to weight function $1_V$ states that 
	  for a simple graph $G$ and nonempty set $S\subseteq E(G)$, we have
	  \begin{equation*}
		X_G=\sum_{\emptyset \neq I\subseteq S}(-1)^{|I|-1}X_{G-I}+(-1)^{|S|}X_{({G/S},
			{(1_V)}_{G/S})}.
	  \end{equation*}
     Note that the term $X_{({G/S}, {(1_V)}_{G/S})}$, under certain conditions, could be eliminated by taking the difference of the chromatic symmetric functions of two graphs.
	For example, the following consequence of Proposition \ref{thm5} generalizes a result of Orellana and Scott \cite[Section 3] {Orellana2014}.

	\begin{coro}\label{coro-de}
		Let $G$ and $H$ be simple graphs, and let  $S\subseteq E(G)$ and $T\subseteq E(H)$ be nonempty, such that  $(G/S,(1_{V(G)})_{G/S})\cong (H/T,(1_{V(H)})_{H/T})$.
		Then
		\begin{equation*}
			X_{G}-X_{H}=
			\sum_{\emptyset \neq I\subseteq S}(-1)^{|I|-1}X_{G-I}
			-\sum_{\emptyset \neq J\subseteq T}(-1)^{|J|-1}X_{H-J}.
		\end{equation*}
	\end{coro}

	Let $G$ be a graph with two incident edges $e_1=uv_1$ and $e_2=uv_2$ such that $v_1$ and $v_2$ are not adjacent.
	Applying Corollary \ref{coro-de} to $G$ and $H:=G-e_1+v_1v_2$ with $S=\{e_2\}$ and $T=\{e_2\}$,
	we obtain
	$$X_G-X_{G-e_1+v_1v_2}=X_{G-e_2}-X_{(G-e_1+v_1v_2)-e_2},
	$$
	which recovers Corollary 3.2 in \cite{Orellana2014}.
	One can also obtain  \cite[Proposition 3.1]{Aliste-Prieto2023} as a special case of Corollary \ref{coro-de}.

		\section{Trees  with path trunks}
		In this section, we consider trees  with exactly two
                vertices of degree at least 3.
                The trunks of these trees are paths, and their twigs are all attached to the ends of these paths. First, we consider the trees whose trunks have exactly two vertices.
		\begin{lem}\label{thm-2tr}
			Let  $T$ be a tree whose trunk has two vertices. If  $F$ is a graph with   $X_{F}=X_{T}$, then $F\cong T$.
		\end{lem}
		\begin{proof} 
			Since $X_T=X_F$, $T$ and $F$ have the same number of vertices and edges (by Proposition \ref{prop-path}).
			Martin et al. \cite[Proposition 3]{Martin2008} showed that the girth of a graph $G$ is determined by $X_G$, so $F$ is also a tree.
			By Lemma \ref{lem-trunk}, the trunk of $F$ also has two vertices,  and $T$ and $F$ have the same twig sequence.
			Let $e$ and $f$ be the unique edges in the trunks of $T$ and $F$, respectively. Let  $T_1,T_2$ be the components of $T-e$, and let $F_1, F_2$ be the components of $F-f$.
			Since $T$ and $F$ have the same twig sequence,  $(T/e,(1_{V(T)})_{T/e})\cong (F/f,(1_{V(F)})_{F/f})$.
			
			Applying Corollary  \ref{coro-de} to $T, \{e\}$ and $F,\{f\}$, we have
			\begin{align*}
				0=X_T-X_F=X_{T-e}-X_{F-f}=X_{T_1} X_{T_2}-X_{F_1}X_{F_2}.
			\end{align*}
			Hence,  by Lemma \ref{thm-forest}, we may assume $X_{T_1}=X_{F_1}$ and $X_{T_2}=X_{F_2}$. Thus, for each $i\in [2]$, $T_i$ and $F_i$ are both spiders or both paths, since they have the same degree sequence.

			Suppose  two or four trees in $\{T_1,T_2,F_1,F_2\}$ are spiders. Without loss of generality, we may assume $T_1$ and $F_1$ are spiders.  Then $T_1\cong F_1$ by Lemma~\ref{lem-trunk}, since $X_{T_1}=X_{F_1}$. 
			Since  $X_T=X_F$, the twig sequences of  $T$ and $F$ are the same (by Lemma \ref{lem-trunk}),  which, together with  $T_1\cong F_1$, implies $T\cong F$.

            Now assume that  $T_1,T_2,F_1,F_2$ are all paths. For $i\in [2]$, let $t_{i1}, t_{i2}$ denote the lengths of the two twigs of $T$ contained in $T_i$, and let $f_{i1},f_{i2}$ denote the lengths of the two twigs of $F$ contained in $F_i$.  
            Then, for $i\in [2]$,  $t_{i1}+t_{i2}=f_{i1}+f_{i2}$ as $X_{T_i}=X_{F_i}$. If $\{t_{11},t_{12}\}=\{f_{11},f_{12}\}$, then $\{t_{21},t_{22}\}=\{f_{21},f_{22}\}$, as $T$ and $F$ have the same twig sequences; hence, $T\cong F$. 
            Now assume $\{t_{11},t_{12}\}\ne \{f_{11},f_{12}\}$, and without loss of generality assume that $t_{11}=\max\{t_{11},t_{12}, f_{11},f_{12}\}$. 
            Then, since  $t_{i1}+t_{i2}=f_{i1}+f_{i2}$, we may assume $t_{11}>f_{11}\ge f_{12}>t_{12}$. 
            This implies that $\{f_{11},f_{12}\}=\{t_{21},t_{22}\}$, and hence $\{t_{11},t_{12}\}=\{f_{21},f_{22}\}$ (as $T$ and $F$ have the same twig sequence). 
            So $T\cong F$.  
		\end{proof}
		
		Martin et al. \cite[Theorem 9]{Martin2008} showed that spiders $T$ are determined by  $X^{(2)}_T:=\sum_{l(\lambda)=2}c_\lambda(T) p_\lambda$.
		We prove the following result by using a similar idea.  Define  $X^{(3)}_T:=\sum_{l(\lambda)=3}c_\lambda(T) p_\lambda$ for a tree $T$.
		\begin{lem}\label{lem-2spider}
			Let $F$ be a forest with exactly two components, $T_1$ and $T_2$, such that $|T_1|=|T_2|$ and both $T_1$ and $T_2$ are spiders.
			Then $F$ is determined by $X_{T_1}^{(3)}+X_{T_2}^{(3)}$, the path sequence of $F$, and the twig sequence of $F$.
			In particular, $F$ is determined by $X_{T_1}+X_{T_2}$
			and the twig sequence of $F$. 
		\end{lem}
		\begin{proof}
			Let $n=|T_1|=|T_2|$. Let $\{s_j\}$ be  the path sequence of $F$, where $s_j$ denotes the number of paths of length $j$ in $F$. By Proposition \ref{prop-path}, $\{s_j\}$ is determined by $X_{T_1}+X_{T_2}$.
			For each $i\in [n]$, let $t_i$ be the number of twigs of length at least $i$ in $F$, and let $x_i$ and $y_i$ be the number of twigs  of length at least $i$ in $T_1$ and $T_2$, respectively.  Then
			\begin{equation*}
				x_i+y_i=t_i.
			\end{equation*}
			It suffices to prove that  $x_i$ and $y_i$ for all $i$ are determined by  $X_{T_1}^{(3)}+X_{T_2}^{(3)}$ and the sequences $\{s_j\}$ and $\{t_j\}$, since $\{t_j\}$ can be determined by the twig sequence of $F$.

			\medskip
			
			\textit {Case} 1.  Every  twig in $F$ has length at most $n/3$.
			
			Let $d_{\lambda}$ denote the absolute value of the coefficient of $p_{\lambda}$ in $X_{T_1}+X_{T_2}$ for all $\lambda\vdash n$. 
			Then, for each $i\le n/3$,  $d_{n-2i,i,i}$ counts the number of choices of two edges from $T_1$ (or $T_2$) whose deletion from $T_1$ (or $T_2$)  results in three components, one of order $n-2i$ (which is at least $n/3$) and the other two each of order $i$. 
			There are two ways to do this.
			One way is to delete two edges from a twig of length at least $2i$ (the $i$-th edge and the $(2i)$-th edge from the leaf in that twig), and the other is to delete one edge  each from two twigs of length at least $i$ (the $i$-th edge). 
			Thus, the number of ways to create three components from $T_1$ (or $T_2$) corresponding to the partition $(n-2i, i, i)$ is $\binom{x_i}{2}+x_{2i}+ \binom{y_i}{2}+y_{2i}$. 
			Hence,
			we have equations in $x_i$ and $y_i$ for $i\le n/3$:
			\begin{align*}
				\begin{cases}
					\binom{x_i}{2}+\binom{y_i}{2}=d_{n-2i,i,i}-(x_{2i}+y_{2i})=d_{n-2i,i,i}-t_{2i};\\
					x_i+y_i=t_i.
				\end{cases}		
			\end{align*}
			Thus, the sets $\{x_i,y_i\}$ for $i\le n/3$ are  determined by these two equations, and hence, by  $X_{T_1}^{(3)}+X_{T_2}^{(3)}$ and the sequence $\{t_j\}$.
			
			To determine $F$ (up to isomorphism), it suffices to determine the ordered pair $(x_i,y_i)$ for all $i\le n/3$. 
			If $x_i=y_i$ for $i\le n/3$ then we are done. 
			So assume there exists $i$  such that $x_i\neq y_i$. 
			Let $k=\max\{i:  x_i\neq y_i\}$. We determine $(x_i,y_i)$ in the order $i=k, k-1, \dots, 1$.
			Since $x_i=y_i$ for all $k<i\le n/3$, we may fix one of the two ordered pairs of $\{x_k,y_k\}$, say  $(x_k,y_k)$.   
			Suppose we have determined the ordered pairs $(x_j,y_j)$ where $j=i+1, \dots, \lfloor n/3\rfloor$ for some $ i< k$.
			
			We count $d_{n-k-i,k,i}$.  Since $n-k-i> n/3$, there are two ways to delete two edges to disconnect $T_1$ (or $T_2$) into components of order $n-k-i, k, i$: delete the $(i+k)$-th edge from a twig of length at least $k+i$  and delete either the $i$-th edge or the  $k$-th edge from the leaf of the twig; or delete the $k$-th edge from one twig of length at least $k$ and the $i$-th edge from another twig of length at least $i$. 
			Thus,  $d_{n-k-i,k,i}=x_k(x_i-1)+y_k(y_i-1)+2t_{k+i}$; so
			\begin{equation*}
				x_kx_i+y_ky_i=d_{n-k-i,k,i}-2t_{k+i}+t_k,
			\end{equation*}
			Since $x_k\neq y_k$ and we also have $x_i+y_i=t_i$ for $i\in [k]$, we see that  $(x_i,y_i)$ is determined by  $X_{T_1}^{(3)}+X_{T_2}^{(3)}$ and the sequence $\{t_j\}$.
			
			\medskip

			\textit {Case} 2. $F$ has twigs of length more than $n/3$.
			
			 Let $L_1, \dots, L_t$ denote the twigs of $F$ with length more than $n/3$,  and let $l_j=|L_j|-1$ for all $j\in [t]$. Without loss of generality, we may assume $l_1\ge \dots\ge l_t$.
			Note that each $T_i$ has at most two twigs of lengths  more than $n/3$; so $t\le 4$.
			Let
			\[s:=\max\{i\colon s_i>0 \},
			\]
			which is the length of a longest path in $F$.
			We now show that $t$ and  $s$ help decide which of the twigs $L_1, \dots, L_t$  belong to the same tree $T_i$:
			\begin{itemize}
				\item 	If $t=1$, then set  $L_1\subseteq T_1$;
				\item if $t=2$, then  set $L_1\cup L_2\subseteq T_1$ (when $s=l_1+l_2$), or  $L_1\subseteq T_1$ and $L_2\subseteq T_2$ (when $s\ne l_1+l_2$);
				\item if $t=3$, then  set  $L_1\cup L_j\subseteq T_1$ (when $s=l_1+l_j$ for some $j\in\{2,3\}$), or set $L_2\cup L_3\subseteq T_1$  (when $s\neq l_1+l_j$ for any $j\in\{2,3\}$);
				\item if $t=4$, then set $L_1\cup L_4\subseteq T_1$ when $s=l_1+l_4=l_2+l_3$ and otherwise set  $L_i\cup L_j\subseteq T_1$ for $\{i,j\}\subseteq [t]$ with $s=l_i+l_j$.
			\end{itemize}
			Thus all ordered pairs $(x_i,y_i)$ for $i>n/3$ are determined by the sequences  $\{t_j\}$ and $\{s_j\}$.
			
			Next, we show that the ordered pairs $(x_i,y_i)$ for $i\le n/3$ can be determined.
			If $x_j=y_j$ for all $j>n/3$, then by a similar argument as in Case 1, we can  determine the ordered pairs $(x_i,y_i)$ for $i\le n/3$ from $X_{T_1}^{(3)}+X_{T_2}^{(3)}$ and the sequence $\{t_j\}$. 
			Let $r=\max\{j:  x_j\neq y_j\}$, and suppose  $r>n/3$. 
			Note that $(x_r,y_r), \ldots, (x_{\lfloor n/3\rfloor+1},y_{\lfloor n/3\rfloor+1})$ are determined above.

			Suppose that  for some  $i\le n/3$, the ordered pairs $(x_r,y_r), \dots, (x_{i+1},y_{i+1})$ have been determined.
			Let $s_j'$ denote the number of paths of length $j$ in $F$ that are each completely contained in a twig of $F$, which is determined by the twig sequence of $F$.
			Then $s_{i+r}-s^\prime_{i+r}$ is the number of paths of length $i+r$ in $F$ that are not contained in any twig of $F$, which  is determined by the path sequence and twig sequence of $F$.

                        For each $k\in \{1, \ldots, \lfloor\frac{i+r}{2}\rfloor\}$, let $p_k$ denote the number of  those paths $P$ in $F$, such that $P$ is contained in $T_j$ for some $j\in [2]$, and $P$ is not contained in any twig of $T_j$, and a subpath of $P$ from one of its ends to the trunk of $T_j$ has length $k$.   
			Then $p_k=x_{i+r-k}(x_k-1)+y_{i+r-k}(y_k-1)$ for $k<(i+r)/2$, and $p_{(i+r)/2}=\binom{x_{(i+r)/2}}{2}+\binom{y_{(i+r)/2}}{2}$ when $i+r$ is even.
			Let
			$$g(i):=\sum_{k=1}^{\lfloor\frac{i+r}{2}\rfloor} (x_{i+r-k} (x_k-1)+y_{i+r-k} (y_k-1).
			$$
		Then
			\begin{align*}
				s_{i+r}-s^\prime_{i+r}=\sum_{k=1}^{\lfloor\frac{i+r}{2}\rfloor}p_k=
				\begin{cases}
					g(i), &\text{if~} i+r \text{~is odd};\\
					g(i)-\binom{x_{(i+r)/2}}{2}-\binom{y_{(i+r)/2}}{2}, &\text{if~} i+r \text{~is even}.
				\end{cases}
			\end{align*}

                        Since $k\le \lfloor \frac{i+r}{2}\rfloor$ and $r>i$, we have $i+r-k\ge \frac{i+r}{2}>i$; so $(x_{i+r-k},y_{i+r-k})$ is determined for all 
                        $k\in \{1, \ldots,\lfloor \frac{i+r}{2}\rfloor\}$ (by assumption). 
                        Moreover, by the definintion of $r$,  we have $x_{i+r-k}=y_{i+r-k}$ which is determined  for each $k\in [i-1]$ (by assumption).  
                        It implies that  
			\[\sum_{k=1}^{i-1} (x_{i+r-k} (x_k-1)+y_{i+r-k} (y_k-1))
			=\sum_{k=1}^{i-1} x_{i+r-k} (t_k-2)
                      \]
                      is determined by the sequences $\{t_j\}$ and $\{s_j\}$.

                      Thus, all terms in the above expression of  $s_{i+r}-s^\prime_{i+r}$, with the exception when $k=i$, are determined by the sequences $\{t_j\}$ and $\{s_j\}$.  
                      Thus, by considering the term when $k=i$, we obtain  the following  equation from expression of  $s_{i+r}-s^\prime_{i+r}$:
			\[x_rx_i+y_ry_i=c_i,
			\]
			where $c_i$ is determined by the sequences $\{s_j\}$ and  $\{t_j\}$. 
			Since  $x_r\neq y_r$ and $x_i+y_i=t_i$, $(x_i,y_i)$  is determined by the sequences $\{s_j\}$ and $\{t_j\}$.
		\end{proof}

		We now prove Theorem \ref{result1} which  extends Lemma  \ref{thm-2tr} to all trees with exactly two vertices of degree at least 3.
		
		\medskip

		\begin{Proof of 1}
			Martin et al. \cite[Proposition 3]{Martin2008} showed that the girth of a graph $G$ is determined by $X_G$, so $F$ is also a tree.
			By  Lemma \ref{lem-trunk} \cite[Theorem 3]{Crew2022},
             $F$ also has exactly two vertices of degree at least 3, and the trunks of $F$ and $T$ are paths of the same length, say $m$.
			We may assume $m\ge 2$ by Lemma~\ref{thm-2tr}. 
			Let $P,Q$ denote the trunks of  $T,F$, respectively. 
			Let $T_1$ and $T_m$ be the components of the induced subgraph obtained from $T$ by  deleting all interior vertices of $P$,  with $|T_1|\ge|T_m|$. 
			Let $F_1$ and $F_m$ be the components of the induced subgraph obtained from $F$ by deleting all interior vertices of $Q$,   with $|F_1|\ge|F_m|$.
                        
			By Lemma \ref{lem-trunk}, the twig sequences of $T$ and $F$ are the same,  since $X_T=X_F$; so we have  $(T/P,(1_{V(T)})_{T/P})\cong (F/Q, (1_{V(F)})_{F/Q})$. Thus,
			by Corollary  \ref{coro-de}, it follows from $X_T=X_F$ that
			\begin{equation}\label{eq26}
				\sum_{\emptyset\neq I\subseteq E(P)}(-1)^{|I|-1}X_{T-I}
				=\sum_{\emptyset\neq J\subseteq E(Q)
				}(-1)^{|J|-1}X_{F-J}.
			\end{equation}
			We will repeatedly apply the differential operator and Lemma \ref{prop3} to \eqref{eq26}. Without loss of generality, we may assume that  $|T_1|\ge|F_1|$.
             
             \begin{cla}
             	When $|T_1|>|T_m|$, we have  $X_{T_m}=X_{F_m}$,  and when $|T_1| = |T_m|$, we have $X_{T_1} +X_{T_m} = X_{F_1} +X_{F_m}$. 
             	As a result, we always have $|T_1|=|F_1|$
             	and $|T_m|=|F_m|$. 
             \end{cla}           
			\begin{proof}
					Let $k=|T_1|+m-1$. 
			Applying $\partial/\partial p_k$
			to \eqref{eq26}, it follows from  Lemma \ref{prop3} that
			\begin{align*}
			&	\sum_{\emptyset\neq I\subseteq E(P)}(-1)^{|I|-1+(k-1)}
				\sum_{
					\substack{
						k\text{-vertex~tree}\\
						H\subseteq T-I}
				}	X_{T-I-V(H)}
				\\=&\sum_{\emptyset\neq J\subseteq E(Q)
				}(-1)^{|J|-1+(k-1)}
				\sum_{
					\substack{
						k\text{-vertex~tree}\\
						L\subseteq F-J}
				}
				X_{F-J-V(L)}.
			\end{align*}
			Each term of the inner summation about $T$ is equal to 0 unless $|I|=1$ and the unique edge in $I$ is the terminal edge of $P$ incident to $T_m$ (when $|T_1 | > |T_m |$) or either of the terminal edges of $P$ (when $|T_1 | = |T_m |$). 
			The same is true for the terms in the inner summation about $F$. 
			Thus the above equation reduces to $X_{T_m} = X_{F_m}$ (when $|T_1| > |T_m|$) or  $X_{T_1} +X_{T_m} = X_{F_1} +X_{F_m}$ (when $|T_1| = |T_m|$).
		\end{proof}

			\medskip

			{\it Case} 1. 	$|T_1|=|T_m|$.
                        
			Then $X_{T_1}+X_{T_m}=X_{F_1}+X_{F_m}$ and $|F_1|=|T_1|=|T_m|=|F_m|$, by the above claim. So by Proposition \ref{prop-path}, $T_1\cup T_m$ and $F_1\cup F_m$ have the same degree sequence and same path sequence.

			Suppose  $T_1$ and $T_m$ are spiders. Then $F_1$ and $F_m$ are also spiders as $T_1\cup T_m$ and $F_1\cup F_m$ have the same degree sequence. 
			By Lemma \ref{lem-trunk} , $T$ and $F$ have the same twig sequence (since $X_T=X_F$). So $T_1\cup T_m$ and $F_1\cup F_m$ also have the same twig sequence.
			Thus, $F\cong T$ by  Lemma \ref{lem-2spider}.

         Now assume exactly one of $T_1,T_m$ is a path, say $T_1$. 
         Then, $T_1\cup T_m$ has precisely one vertex of degree at least 3, and $T_m$ is a spider. 
         Hence, $F_1\cup F_m$ also has exactly one vertex of degree at least 3, as $T_1\cup T_m$ and $F_1\cup F_m$ have the same degree sequence. 
         Therefore, precisely one of $F_1, F_m$ is a path, and we may assume that $F_1$ is a path and $F_m$ is a spider. 
         Then, $T_1\cong F_1$ as $|T_1|=|F_1|$. 
         Since    $X_{T_1}=X_{F_1}$ (as $T_1\cong F_1$) and $X_{T_1}+X_{T_m}=X_{F_1}+X_{F_m}$, we have $X_{T_m}=X_{F_m}$. 
         So $T_m\cong F_m$  by Lemma~\ref{lem-trunk}. 
         Since the twig sequences of $T$ and $F$ are the same, the lengths of the two twigs of $T$ contained in $T_1$ are the same as those of the two twigs of $F$ contained in $F_1$.  Therefore,  $F\cong T$.
         
         It remains to consider the case when both  $T_1$ and $T_m$ are paths. Then $F_1$ and $F_m$ are also paths as $T_1\cup T_m$ and $F_1\cup F_m$ have the same degree sequence.
         For $i\in \{1,m\}$, let $t_{i1}, t_{i2}$ denote the lengths of the two twigs of $T$ contained in $T_i$, and let $f_{i1},f_{i2}$ denote the lengths of the two twigs of $F$ contained in $F_i$.  
         Then, $t_{11}+t_{12}=t_{m1}+t_{m2}=f_{11}+f_{12}=f_{m1}+f_{m2}$. 
         If $\{t_{11},t_{12}\}=\{f_{11},f_{12}\}$, then $\{t_{m1},t_{m2}\}=\{f_{m1},f_{m2}\}$, as $T$ and $F$ have the same twig sequence; hence, $T\cong F$. 
         Now assume $\{t_{11},t_{12}\}\ne \{f_{11},f_{12}\}$, and without loss of generality assume that $t_{11}=\max\{t_{11},t_{12}, f_{11},f_{12}\}$. 
         Then, since  $t_{11}+t_{12}=f_{11}+f_{12}$, we may assume $t_{11}>f_{11}\ge f_{12}>t_{12}$. 
         This implies that $\{f_{11},f_{12}\}=\{t_{m1},t_{m2}\}$, and hence $\{t_{11},t_{12}\}=\{f_{m1},f_{m2}\}$ (as $T$ and $F$ have the same twig sequence). 	Thus, we have  $F\cong T$.
			
			\medskip
			
            {\it Case} 2.  $|T_1|>|T_m|$.
                        
             Then $X_{T_m}=X_{F_m}$ from the above claim. By Proposition~\ref{prop-path} and Lemma~\ref{lem-trunk}, $T_m$ and $F_m$ have the same degree sequence, path sequence, and twig sequence. 
             Recall that  $T$ and $F$ have the same twig sequence and isomorphic trunk (by  Lemma \ref{lem-trunk} since $X_T=X_F$). 
                        
			If $T_m$ and $F_m$ are spiders, then $T_m\cong F_m$; hence, since $T$ and $F$ have the same twig sequence and isomorphic trunk, $T\cong F$.
                        So we may assume that both $T_m$ and $F_m$ are paths. 

                        Now suppose that $T_1$ and $F_1$ are also  paths.
                        For $i\in \{1,m\}$, let $t_{i1}, t_{i2}$ denote the lengths of the two twigs of $T$ contained in $T_i$, and let $f_{i1},f_{i2}$ denote the lengths of the two twigs of $F$ contained in $F_i$.  Then, $t_{11}+t_{12}=f_{11}+f_{12}>f_{m1}+f_{m2}=t_{m1}+t_{m2}$. If $\{t_{11},t_{12}\}=\{f_{11},f_{12}\}$, then $\{t_{m1},t_{m2}\}=\{f_{m1},f_{m2}\}$, as $T$ and $F$ have the same twig sequences; hence, $T\cong F$. Now assume $\{t_{11},t_{12}\}\ne \{f_{11},f_{12}\}$, and without loss of generality assume that $t_{11}=\max\{t_{11},t_{12}, f_{11},f_{12}\}$. Then, since  $t_{11}+t_{12}=f_{11}+f_{12}$, we may assume $t_{11}>f_{11}\ge f_{12}>t_{12}$. This implies that $\{f_{11},f_{12}\}=\{t_{m1},t_{m2}\}$ (as $T$ and $F$ have the same twig sequence), contradicting $f_{11}+f_{12}>t_{m1}+t_{m2}$.

			  It remains to consider the case when  $T_1$ and $F_1$ are spiders, and $T_m$ and $F_m$ are paths.
			 Let $e_1$ be the edge of $P$ incident with $T_1$, let $e_m$ be the edge of $P$ incident with $T_m$,  and let $e_{m-1}$ be the edge of $P-e_m$ incident with $e_m$.  We use $T_m+e_m$ to denote the subtree of $T$ induced by  $E(T_m)\cup \{e_m\}$. Define  $f_1, f_m,f_{m-1}$ and $F_m+f_m$ similarly.
                          \medskip

                          {\it Subcase} 2.1.   $|T_1|>|T_m|+1$.

                          Then for any nonempty $I\subseteq E(P)$, $T-I$ contains a subtree of order $|T_1|+m-2$ if and only if $I\subseteq \{e_{m-1},e_m\}$.
                           Therefore, 
                          by applying  $\partial/\partial p_{|T_1|+m-2}$ to \eqref{eq26}, it follows from Lemma~\ref{prop3} that
			$$
			\frac{\partial(X_{T-{e_m}}+X_{T-e_{m-1}}-X_{T-\{e_m,e_{m-1}\}})
			}{\partial p_{|T_1|+m-2}}
			=\frac{\partial(X_{F-{f_m}}+X_{F-f_{m-1}}-X_{F-\{f_m,f_{m-1}\}})
			}{\partial p_{|T_1|+m-2}}.$$
                        This yields
			$$
			(l_1+1)p_1X_{T_m}+X_{T_m+e_m}
			-p_1X_{T_m}
			=(l_2+1)p_1X_{F_m}+X_{F_m+f_m}
			-p_1X_{F_m},$$
			where $l_1$ and $l_2$ denote the numbers of leaves of $T_1$ and $F_1$, respectively.
                        
			Note that  $l_1=l_2$, since by Proposition \ref{prop-path}, $T$ and $F$ have the same number of leaves (as $X_T=X_F$) and $T_m$ and $F_m$ have the same number of leaves (as
                        $X_{T_m}=X_{F_m}$).  Since $X
			_{T_m}=X_{F_m}$, we have $X_{T_m+e_m}=X_{F_m+f_m}$. Hence by Lemma~\ref{lem-trunk}, $T_m+e_m$ and $F_m+f_m$ (as spiders) have the same twig sequence.           Thus, the lengths of two twigs of $T$ contained in $T_m$ are the same as the lengths of the two twigs of $F$ contained in $F_m$.   Hence, $T\cong F$ as $T$ and $F$ have the same twig sequence and isomorphic trunk.

			\medskip
                        
			{\it Subcase} 2.2.  $|T_1|=|T_m|+1$.

                        Then for any nonempty $I\subseteq E(P)$, $T-I$ contains a subtree of order $|T_1|+m-2$ if and only if $I\subseteq \{e_{m-1},e_m\}$ or $I=\{e_1\}$.
                        As in Subcase 2.1, we apply  $\partial/\partial p_{|T_1|+m-2}$ to \eqref{eq26}.  It follows from Lemma~\ref{prop3} and $X_{T_m}=X_{F_m}$ that 
			\begin{equation*}
				X_{T_m+e_m}+X_{T_1}=X_{F_m+f_m}+X_{F_1}.
                              \end{equation*}
                              
                         Note that  the twig sequence of $T_1\cup (T_m+e_m)$ is obtained from the twig sequence of $T$ by adding one to the number of twigs of length 1 in $T$
                         (corresponding to the twig of $T_m+e_m$  induced by $e_m$) and the twig sequences of $F_1\cup (F_m+f_m)$ is obtained from the twig sequence of $F$ by adding one to the number of twigs of length 1 in $F$
                         (corresponding to the twig of $F_m+f_m$ induced by $f_m$). Therefore, the twig sequence of  $T_1\cup (T_m+e_m)$ and  $F_1\cup (F_m+f_m)$ are the same (as $X_T=X_F$).
			It follows from Lemma \ref{lem-2spider} that
			$T_1\cup (T_m+e_m)\cong F_1\cup (F_m+f_m)$. Hence,  $T_1\cong F_1$ and  $T_m+e_m\cong F_m+f_m$,
			or $T_1\cong  F_m+f_m$ and $T_m+e_m\cong F_1$.  
			
			Suppose  $T_1\cong F_1$ and  $T_m+e_m\cong F_m+f_m$. 
			Since $T_1$ and $F_1$ are spiders,  the lengths of twigs of $T$ contained
			in $T_1$ are the same as those of the twigs of $F$ contained in $F_1$. Similarly, since $T_m+e_m$ and $F_m+f_m$ are spiders, we see that the lengths of the
			two twigs of $T$ contained in $T_m$ are the same as those of the two twigs of $F$ contained in $F_m$. Thus, since $T$ and $F$ have isomorphic trunk, we
			must have $T\cong F$.
			
			Now assume $T_1\cong  F_m+f_m$ and $T_m+e_m\cong F_1$. 
			Then $T_1$ has exactly 3 twigs and one of these is of length 1 (and we use $u$ to denote its leaf), and if we use $t_{11}, t_{12}$ to denote the lengths of the other two twigs in $T_1$,  then $t_{11}, t_{12}$ are the lengths of the two twigs of $F$ contained in $F_m$.  Similarly, $F_1$ has exactly 3 twigs and one of them is of length 1 (and we use $v$ to denote its leaf), and if we use $f_{11}, f_{12}$ to denote the lengths of the other two twigs in $F_1$, then $f_{11}, f_{12}$ are the lengths of the two twigs of $T$ contained in $T_m$. Note that $T-u\cong F-v$. 
			Also note that  if $\{t_{11},t_{12}\}=\{f_{11},f_{12}\}$, then $T\cong F$.
			
			So we may assume $\{t_{11},t_{12}\}\ne \{f_{11},f_{12}\}$. Since $t_{11}+t_{12}=f_{11}+f_{12}$ (as $|T_1|=|F_1|$), we may further assume that  $t_{11}>f_{11}\ge f_{12}>t_{12}$. We now derive a contradiction by counting the paths of length $t_{11}+m+1$ in each of   $T$ and $F$. Since $T-u\cong F-v$, $T-u$ and $F-v$ have the same number of paths of length $t_{11}+m+1$. Therefore,  since $T$ and $F$ have the same path sequence, the number of paths of length $t_{11}+m+1$ in $T$ containing $u$ must equal the number of paths of length $t_{11}+m+1$ in $F$ containing $v$. But this is not true, as no such path exists in $T$, and there is one such path in $F$. 
                    \end{Proof of 1}
\begin{rem}
Theorem \ref{result1} confirms Conjecture \ref{conj} for trees with exactly two vertices of degree at least 3. Our approach may be extended to confirm
Conjecture \ref{conj} for trees with exactly 3 vertices of degree at least 3, as trunks of those trees are also paths. In general, it is important to know whether the trunks of trees can be determined from their chromatic symmetric functions. 
\end{rem}

	\section{The generalized degree sequence}
		
		There is an interesting problem posed by Crew in \cite{Crew2022}.
		In  \cite{Crew2020-1, Crew2022},
		Crew introduced
		the {\it generalized} degree sequence of a graph $G$ as  the multiset
		$$\{(|W|,e(W),d(W)) \colon W\subseteq V(G)\},$$
		where $e(W)=|E(G[W])|$ and $d(W)$ denotes the number of edges of $G$
		with exactly one endpoint in $W$. When restricting this definition to those $W$ with
		$|W|=1$, one obtains the degree sequence of $G$.
		\begin{conj}
			[Crew, \cite{Crew2022}] For  any tree $T$, the generalized degree sequence of $T$ is determined by $X_T$.
		\end{conj}
		
		We will show that $X_T$ determines the multiset $\{(|H|,d(V(H))): H \mbox{ a subtree of } G\}$. 
		Define 
		$$F_T=F_T(x,y)=\sum_{ H}x^{|H|}y^{d(V(H))},$$
where the summation is taken over all subtrees $H$ of $T$.  Note that, for any subtree $H$ of $T$,
$d(V(H))$ is the number of components of $T-V(H)$.

		Let $f_{T}(i,j)$ be  the coefficient of $x^iy^j$ in the polynomial $F_T(x,y)$.
                Then $f_T(1,j)$ is the number of vertices of degree $j$ in $T$. Recall that $X_T=\sum_{\lambda\vdash n}c_{\lambda}(T)p_{\lambda}$.  We will show that $$f_T(i,j) =\sum_{\lambda\vdash n}\sigma(\lambda,i,j)c_\lambda(T),$$
                where, for a partition $\lambda=(\lambda_1,\dots,\lambda_l)\vdash n$ and integers $i,j$, 
		\begin{equation}\label{eq-sigma}
		\sigma(\lambda,i,j)
		=(-1)^{n-j-1}
		\binom{n-i-l(\lambda)+1}{j-l(\lambda)+1}
		\left(\sum_{s=1}^{l(\lambda)}
		\delta(\lambda_s,i)\right),			
	\end{equation}
		and  $\delta(j,i)=1$ (when $j=i$) or $\delta(j,i)=0$ (when $j\ne i$). Note that $\sum_{s=1}^{l(\lambda)}		\delta(\lambda_s,i)$ is the number of parts of size $i$ in the partion $\lambda$. 
		
		\begin{thm}\label{thm-delete}
			For every tree $T$ of order $n$,
			$\displaystyle F_T(x,y)=\sum_{i=1}^n\sum_{j=0}^{n-i}x^iy^j\sum_{\lambda\vdash n}\sigma(\lambda,i,j)c_\lambda(T).
			$
		\end{thm}
		\begin{proof}
			Applying Lemma \ref{prop3}, we have that for every $i\le|V(T)|$,
			\begin{equation}\label{eq23}
				\frac{\partial X_T}{\partial p_{i}}=
				(-1)^{i-1}
				\sum_{H} X_{T-V(H)}		
			\end{equation}
			where the summation is taken over all $i$-vertex subtrees $H$ of $T$. 
			Expand each $X_{T-V(H)}$ of $\partial X_T/\partial p_i$  in the basis
			$\{p_{\lambda^\prime}\}$ of $\varLambda_{n-i}$, and  let $d_{i,k}$ denote the sum of coefficients of $p_{\lambda^\prime}$ with $l(\lambda^\prime)=k$ in $(-1)^{i-1}X_{T-V(H)}$ for all $i$-vertex trees $H$ in $T$.          Since $\sum_{s=1}^{k+1}\delta(\lambda_s,i)$ is the number of parts of size $i$ in $\lambda=(\lambda_1,\dots,\lambda_{k+1})$, we can express $d_{i,k}$ with $c_\lambda(T)$: $$d_{i,k}=\sum_{\lambda=(\lambda_1,\cdots,\lambda_{k+1}) \vdash n
	}
	\left(\sum_{s=1}^{k+1}\delta(\lambda_s,i)\right)c_{\lambda}(T).$$

        Recall that $f_{T}(i,j)$, the coefficient of $x^iy^j$ in the polynomial $F_T(x,y)$, is the number of $i$-vertex subtrees of $T$ such that $T-V(H)$ has exactly $j$ components.
				For every $i$-vertex subtree $H$ of $T$ such that $T-V(H)$ has exactly $j$ components, if $k\ge j$ (respectively, $k<j$), then  $(-1)^{n-i-k}\binom{n-i-j}{k-j}$ (respectively, 0) is the sum of coefficients of $p_{\lambda^\prime}$
			with $l(\lambda^\prime)=k$ in $X_{T-V(H)}$ (see Section 2). 
			Thus, by \eqref{eq23} for every $k\ge 1$,
			\[d_{i,k}=(-1)^{(i-1)+(n-i-k)}\sum_{j=1}^{k} \binom{n-i-j}{k -j}f_T(i,j).
			\]
                        Observe that  this expression is a recursive relation about $d_{i,k}$ and $f_T(i,j)$ with respect indices $k$ and $j$. 
                        Let $A_{k,k}=(a_{m,j})$ be a $k\times k$-matrix
                        with $a_{m,j}=(-1)^{n-m-1}\binom{n-i-j}{m-j}$ for $m\ge j$, and $a_{m,j}=0$ for $m<j$. Then 
                         the above recursive relation can be written as  $D_k=A_{k,k}F_k$,  with column vectors $D_k=(d_{i,m})_{1\le m\le k}$ and $F_k=(f_T(i,j))_{1\le j\le k}$. One can verify that for any integer $k$, $A_{k,k}^2=I$, the $k\times k$ identity matrix, so we obtain $F_k=A_{k,k}D_k$, which implies
			\begin{align*}
			f_T(i,j)=(-1)^{n-j-1}\sum_{k=1}^{j}\binom{n-i-k}{j-k}d_{i,k}.
		\end{align*}
Therefore, 
	\begin{align*}
		f_T(i,j)&=
				(-1)^{n-j-1}\sum_{k=1}^{j}\binom{n-i-k}{j-k}\left(
				\sum_{\lambda=(\lambda_1,\cdots,\lambda_{k+1}) \vdash n
				}
			       \left(\sum_{s=1}^{k+1}\delta(\lambda_s,i)\right)c_{\lambda}(T)\right)\\
				&=(-1)^{n-j-1}
				\sum_{\lambda\vdash n}
				\binom{n-i-(l(\lambda)-1)}{j-(l(\lambda)-1)}
				\left(\sum_{s=1}^{l(\lambda)}
				\delta(\lambda_s,i)\right)c_{\lambda}(T)\\
				&=\sum_{\lambda\vdash n}\sigma(\lambda,i,j)c_\lambda(T),
        \end{align*}
        completing the proof.
		\end{proof}

		Crew \cite[Lemma 26]{Crew2020-1}  proved that for a tree $T$, both
		the order of its trunk and the twig sequence are  determined by
		$F_T$,
		and
		thus also by $X_T$ by applying  Theorem  \ref{thm-delete}.
		Therefore, Theorem \ref{thm-delete} gives another proof of Lemma \ref{lem-trunk}
		via the polynomial $F_T$,
		instead of the subtree polynomial $S_T$.
		
		Similar to a statement  in \cite[Page 246]{Martin2008} about the subtree polynomial $S_T$ (see Section 2), we obtain a formula for $F_T$ in terms of the usual scalar product $\langle\cdot,\cdot\rangle$ on $\varLambda_n$ (see \cite[\S7.9]{Stanley1999}), where $n=|T|$.  Define
		$$\Omega_n(x,y)=\sum_{i=1}^n\sum_{j=0}^{n-i}x^iy^j\sum_{\lambda\vdash n}\sigma(\lambda,i,j)\frac{p_\lambda}{\langle p_\lambda,p_\lambda\rangle}.$$
		Then Theorem \ref{thm-delete} is equivalent to the statement:  For every tree $T$ of order $n$,
		$F_T(x,y)=\langle \Omega_n(x,y),X_T\rangle.$

		
		

                \section*{Acknowledgements}
		The authors are grateful to the referees  for their valuable comments and suggestions. In  particular,  we thank the referee who suggested the possibility of  an explicit linear transformation from $X_T$ to $F_T$, resulting in the current form of Theorem \ref{thm-delete}. We also thank the referees and Sagar Sawant  for  pointing  out that Lemma \ref{thm-forest} is a corollary of \cite[Corollary 2.4]{Tsujie2018}.

	\end{document}